\documentclass[11pt, reqno]{amsart}
\usepackage{cite}
\usepackage{wrapfig, lipsum,booktabs}
\usepackage{graphicx}
\usepackage{amssymb}
\usepackage{epstopdf}
\usepackage{verbatim}
\usepackage{bm}
\usepackage{multicol}
\usepackage{multirow}
\usepackage{subfigure}
\usepackage{fancyhdr} 
\usepackage{float}
\usepackage{stmaryrd}
\usepackage{color}
\usepackage{soul}
\usepackage{tikz}
\usepackage{todonotes}
\usetikzlibrary{patterns}
\usepackage{pgfplots}
\usepackage{cancel}
\newtheorem{theorem}{Theorem}[section]

\newtheorem{remark}{Remark}[section]

\usepackage{multirow}
\usepackage{amsmath,amssymb,eucal}
\usepackage{graphicx,subfigure,epsfig}
\usepackage{psfrag}
\usepackage{url}
\usepackage[top=1in, bottom=1.in, left=1in, right=1in]{geometry}

\newcommand{\Th}{\mathcal{T}_h}
\newcommand{\pol}{\mathcal{P}}

\setulcolor{red}

\begin{document}
\title[STFEM-PNP]{High-order
space-time finite element methods for the Poisson-Nernst-Planck equations:
positivity and unconditional energy stability}
\author{Guosheng Fu}
\address{Department of Applied and Computational Mathematics and 
Statistics, University of Notre Dame, USA.}
\email{gfu@nd.edu}
\author{Zhiliang Xu}
\address{Department of Applied and Computational Mathematics and 
Statistics, University of Notre Dame, USA.}
\email{zhiliangxu@nd.edu}
 \thanks{G. Fu acknowledge the partial support of this work
 from U.S. National Science Foundation through grant DMS-2012031.
 Z. Xu was partially supported by the NSF CDS\&E-MSS 1854779.
 }

\begin{abstract}
  We present a novel class of high-order
  space-time finite element schemes for the Poisson-Nernst-Planck (PNP)
  equations.  We prove that our schemes are mass conservative, positivity preserving, and
 unconditionally  energy stable for any order of approximation.
To the best of our knowledge, this is the first class of (arbitrarily) high-order
  accurate schemes for the PNP equations that simultaneously achieve all these
three  properties.

This is accomplished via (1) 
  using finite elements to directly approximate 
the so-called  {\it entropy variable} $u_i :=U'(c_i)=\log(c_i)$
instead of the density variable $c_i$, where 
  $U(c_i)=(\log(c_i)-1) c_i$ is the corresponding entropy, and 
  (2) using a discontinuous Galerkin (DG) discretization in time.
The {\it entropy variable} formulation, which was originally developed by Metti et al. \cite{Metti16} under the name of a {\it log-density formulation},
 guarantees both positivity of densities
$c_i = \exp(u_i)>0$ and a continuous-in-time
energy stability result.
The DG in time discretization further ensures an
unconditional energy stability in the fully discrete level 
for any approximation order, where the lowest order case is exactly the backward
Euler discretization and in this case we recover the method of Metti et al. \cite{Metti16}.
\end{abstract}
\maketitle

\section{Introduction}
\label{sec:intro}
The Poisson-Nernst-Planck (PNP) equations describe the diffusion of charged
particles under the effect of an electric field that is itself affected by these
particles. This system of equations has been widely used in the modelling of
semiconductors \cite{Jerome96} and ion channels in biology \cite{Hille01}.

Various numerical methods with different properties have been developed for the
PNP equations
\cite{Bank83, Prohl09, Lu10, Zheng11, Flavell14, Meng14}. 
We particularly cite the very recent schemes 
\cite{Metti16, 
  Hu20, Liu20, ShenXu20,
LiuWang20, HuangShen21} that were provably 
positivity preserving for the  particle densities and 
unconditionally energy dissipative for the free energy, among which
the schemes in \cite{Metti16, Hu20, Liu20, ShenXu20,
LiuWang20} are first-order accurate in time, while the scheme in \cite{HuangShen21}
is  high-order accurate in time but the associated
energy dissipation law is only valid for a {\it modified energy} due to the use of
the recent scalar auxiliary variable (SAV) technique \cite{ShenXuYang19}.
%
To the best of our knowledge, no provable positivity
preserving and unconditionally energy dissipative scheme for the {\it original
energy} that is at least
second-order accurate in time exists so far. 
We fill this gap by presenting a class of
arbitrarily high-order accurate  space-time finite element (STFEM) schemes satisfying these
properties. Our spatial discretization is the same as the one developed in 
\cite{Metti16} which is
based on a {\it log-density formulation}.
%
We remark that this log-density based formulation 
is similar in spirit to the entropy-stable schemes based on 
the so-called {\it entropy
variables} for hyperbolic 
conservation laws and compressible flow in the CFD literatures
\cite{Harten83, Tadmor84,Hughes86,Barth99}.
Our temporal discretization is an upwinding discontinuous Galerkin (DG) scheme,
which guarantees unconditional energy stability.
We also discuss adaptive time stepping using the classical proportional integral
(PI) step size controller \cite{Gustafsson88}.

Since the PNP equations can be viewed as a  Wasserstein gradient flow
\cite{Kinderlehrer17},
we expect our STFEM scheme with {\it entropy variables}  can be applied to 
construct high-order positivity preserving and unconditionally energy stable schemes for other
Wasserstein gradient flow problems, like the Fokker-Planck equation and the
porous medium equation.

The rest of the paper is organized as follows.
In Section \ref{sec:method}, we first introduce the PNP equations then 
present the spatial/temporal finite element discretizations and prove that they
are mass conservative, positivity preserving, and unconditionally energy dissipative.
We further discuss about the nonlinear system solver via the Newton's method and
adaptive time step size control.
Numerical results are presented in Section \ref{sec:num}.
We conclude in Section \ref{sec:conclude}.

\section{The PNP equations and the space-time finite element schemes}
\label{sec:method}
\subsection{PNP equations}
We consider the PNP equations with $N$ species of charged particles \cite{Eisenberg98}
on a bounded domain $\Omega\subset \mathbb{R}^d$, $d=1,2,3$, with 
boundary $\partial\Omega$:
\begin{subequations}
  \label{pnp}
  \begin{align}
  \label{pnp1}
    \frac{\partial c_i}{\partial t} = &\; \nabla\cdot \left(
  D_i c_i\nabla\mu_i
  \right),\quad i=1,2,\cdots, N,\\
  \label{pnp2}
      -\nabla\cdot(\epsilon\nabla \phi) = &\; \rho_0+\sum_{i=1}^Nz_iec_i,
  \end{align}
where $c_i$ is the density of the $i$-th charged particle species, 
$D_i$ is the diffusion
constant, 
$$
\mu_i := 
 \log(c_i)
   +\frac{z_i e}{k_B T}\phi
$$
is the chemical potential of the $i$-th species, 
$z_i$ is the valence, $e$ is the unit charge, $k_B$ is the Boltzmann
constant, $T$ is the absolute temperature, $\epsilon$ is the electric permittivity, $\phi$
is the electrostatic potential, $\rho_0$ is the permanent (fixed) charge density of the
system, and $N$ is the number of charged particle species.

The PNP system \eqref{pnp1}--\eqref{pnp2} is closed with the following set of initial and
boundary conditions:
\begin{alignat}{2}
  \label{pnp3}
  c_i(0,x) =&\; c_i^0(x),&&\quad \text{ in }\Omega \\
  \label{pnp4}
  c_i\frac{\partial \mu_i}{\partial n} = 
 \frac{\partial \phi}{\partial n}   =&\; 0,  &&\quad \text{ on }\partial\Omega.
\end{alignat}
\end{subequations}
Here for simplicity, we use the homogeneous Neumann boundary condition for 
the charged particle densities.
Other boundary conditions will be used in our numerical experiments.
Note that the electrostatic potential $\phi$ is determined up to a constant due
to the pure Neumann boundary condition. 
Following the classical convention, we select a unique $\phi$ by  requiring 
$$
\int_{\Omega}\phi\,\mathrm{dx}=0.
$$

The PNP system \eqref{pnp} satisfies the following three important properties:
\begin{subequations}
  \label{prop} 
\begin{itemize}
  \item [(i)] Mass conservation: 
    \begin{align}
      \label{mass}
      \int_{\Omega}c_i(t, x)\mathrm{dx} = 
    \int_{\Omega}c_i^0(x)\mathrm{dx}. 
    \end{align}
  \item [(ii)] Positivity preservation: 
\begin{align}
  \label{positivity}
  \text{
  If $c_i^0(x)>0$, then $c_i(t, x)>0$ for any $t>0$.}
\end{align}
  \item [(iii)] Energy dissipation:
    \begin{align}
      \label{ener}
      \frac{d}{dt} E=-\sum_{i=1}^N\int_{\Omega}D_ic_i|\nabla\mu_i|^2\mathrm{dx},
    \end{align}
    where 
    $
    E(\{c_i\}, \phi):=
    \int_{\Omega}\left(
      \sum_{i=1}^N(c_i(\log(c_i)-1))+\frac12\frac{\epsilon}{k_BT}|\nabla\phi|^2
    \right)\;\mathrm{dx}
    $
    is the total free energy.
\end{itemize}
\end{subequations}

\begin{remark}
The above energy dissipation \eqref{ener} is
obtained by a standard energy variational argument, where, in particular, one 
multiplies the equation \eqref{pnp1} with the test function
$\mu_i=\log(c_i)+\frac{z_i e}{k_BT}\phi$ and integrates over the domain $\Omega$.
An immediate consequence is that such energy dissipation would fail to hold in
a standard finite element discretization where one only approximates the fields $c_i$ and 
$\phi$ using finite elements.
\end{remark}

One approach to recover energy stability
is to use a log-density formulation \cite{Metti16}, where one 
directly discretizes the {\it entropy variables}
\begin{align}
  \label{entropy}
  u_i = U'(c_i) = \log(c_i),
\end{align}
instead of the densities $c_i$, 
  where $U(c_i):= c_i(\log(c_i)-1)$  is the entropy of the $i$-th species.
  In this formulation, 
  the species densities $c_i(u_i):=\exp(u_i)$ are {\it derived}
  variables which are guaranteed to stay positive $c_i=\exp(u_i)>0$ for any
  time $t>0$.
  We point out that such {\it entropy-variable} based schemes 
  are similar in spirit to the
  entropy stable schemes using 
  {\it entropy variables}
  for  hyperbolic conservation laws and compressible flow \cite{Harten83,
  Tadmor84, Hughes86, Barth99}.

  To this end, we work with the following reformulated  PNP equations
  \eqref{pnp1}--\eqref{pnp2}:
  \begin{subequations}
    \label{pnpY}
    \begin{align}
      \label{pnpY1}
    \frac{\partial \exp(u_i)}{\partial t} = &\; \nabla\cdot \left(
      D_i\exp(u_i)\left(\nabla u_i
   +\frac{z_i e}{k_B T}\nabla\phi
 \right)
  \right),\\
      -\nabla\cdot(\epsilon\nabla \phi) = &\; \rho_0+\sum_{i=1}^Nz_ie\exp(u_i).
  \end{align}
  \end{subequations}

\newcommand{\Vh}{V_h^{k}}
\newcommand{\Vho}{V_{h,0}^{k}}

\newcommand{\Vst}{\mathsf{V}_h^{k,m,n}}
\newcommand{\Vstx}{\mathsf{V}_{h}^{k,m-1,n}}
\newcommand{\Vsto}{\mathsf{V}_{h}^{k,m,n}}

\subsection{Spatial discretization}
Let $\Th:=\{K\}$ be a conforming simplicial
triangulation
 of the domain $\Omega$.
 We shall use the following conforming finite element space
 \begin{align}
   \label{space-pk}
   \Vh := \{v\in H^1(\Omega):\; v|_K\in\pol_k(K),\quad \forall K\in\Th\},
 \end{align}
 where $\pol_k(K)$ is the space of polynomials of degree at most $k\ge 1$ on $K$.

 The spatial discretization of our finite element scheme 
for the equations \eqref{pnpY} with initial and boundary conditions 
\eqref{pnp3}--\eqref{pnp4}
reads as follows:
Find $u_{h,1}, \cdots, u_{h,N}\in V_h$ and $\phi_h\in V_{h}$
with $\int_{\Omega}\phi_h\,\mathrm{dx}=0$
such that, for $t>0$, 
\begin{subequations}
  \label{fem}
  \begin{align}
    \label{fem1}
    \int_{\Omega}\frac{\partial \exp(u_{h,i})}{\partial t} v_i\,\mathrm{dx} 
    +\int_{\Omega} 
    D_i\exp(u_{h,i})\left(\nabla u_{h,i}
   +\frac{z_i e}{k_B T}\nabla\phi_h
 \right)\cdot\nabla v_i\,\mathrm{dx}
=&\;0,\quad 
\forall v_i\in V_h,\\
    \label{fem2}
\int_{\Omega}
\epsilon\nabla \phi_h\cdot\nabla\psi\,\mathrm{dx} -\; 
\int_{\Omega}
\left(\rho_0+\sum_{i=1}^Nz_ie\exp(u_{h,i})\right)
\psi\,\mathrm{dx} =&\;0,\quad \forall \psi\in V_{h},
  \end{align}
\end{subequations}
with initial conditions
$$
u_{h,i}(0, x) = \log(c_i^0(x)),\quad i=1,2,\cdots, N.
$$

The following results show that 
the semi-discrete scheme \eqref{fem} satisfies the three properties
\eqref{prop}.

\begin{theorem}
  \label{thm:fem}
  Assume $c_i^0>0$ for all $i$. 
Then the three properties \eqref{prop} are satisfied 
for the solution to the semi-discrete scheme \eqref{fem}, 
where the densities $c_i$ in \eqref{prop} 
are given explicitly as $c_i=\exp(u_{h,i})$.
\end{theorem}
\begin{proof}
  Taking $v_i=1$ in \eqref{fem1}, we get mass conservation property
  \eqref{mass}.
  Positivity property \eqref{positivity} following directly by the definition of
  $c_i=\exp(u_{h,i})>0$.

  Let us prove the energy dissipation property \eqref{ener}.
  Denoting $\mu_{h,i}:=u_{h,i}+\frac{z_i e}{k_B T}\phi_h\in V_h$, 
  taking test function $v_i=\mu_{h,i}$ in \eqref{fem1}
  and adding, we get
  \begin{align*}
    \sum_{i=1}^N\int_{\Omega}\frac{\partial \exp(u_{h,i})}{\partial t} 
    \mu_{h,i}
    \,\mathrm{dx} 
    +
    \sum_{i=1}^N
    \int_{\Omega} 
 D_i\exp(u_i)|\nabla\mu_{h,i}|^2 \,\mathrm{dx}
=&\;0.
  \end{align*}
 A simple calculation yields that 
  \begin{align*}
    \int_{\Omega}
    \frac{\partial \exp(u_{h,i})}{\partial t} 
    u_{h,i}
    \,\mathrm{dx} 
    =&\; \frac{d}{dt}
    \int_{\Omega}\underbrace{\exp(u_{h,i})(u_{h,i}-1)}_{:=U(u_{h,i})}
    \,\mathrm{dx}. 
  \end{align*}
  Hence, 
  \begin{align}
    \label{ener1}
    \sum_{i=1}^N
\frac{d}{dt}
    \int_{\Omega}U(u_{h,i})
    \,\mathrm{dx} 
    +
    \int_{\Omega}
    \sum_{i=1}^N
    \frac{z_i e}{k_BT} \frac{\partial\exp(u_{h,i})}{\partial t} 
   \phi_h \,\mathrm{dx} 
    +
    \sum_{i=1}^N
    \int_{\Omega} 
 D_i\exp(u_i)|\nabla\mu_{h,i}|^2 \,\mathrm{dx}
=&\;0.
  \end{align}
  Taking $\psi= \frac{\partial \phi}{\partial t}/({k_BT})$ in equation \eqref{fem2} and subtract the
  resulting expression from \eqref{ener1}, we get
  \begin{align*}
\frac{d}{dt}
E_h
    +
    \sum_{i=1}^N
    \int_{\Omega} 
 D_i\exp(u_i)|\nabla\mu_{h,i}|^2 \,\mathrm{dx}
=&\;0,
  \end{align*}
where the discrete free energy
\begin{align*}
  E_h:= \int_{\Omega}\left(\sum_{i=1}^NU(u_{h,i})+
    \sum_{i=1}^N
    \frac{z_i e}{k_BT} \exp(u_{h,i})
   \phi_h 
   -\frac{\epsilon}{2k_BT}|\nabla\phi_h|^2
 +\frac{\rho_0}{k_BT}\phi_h\right)
  \mathrm{dx}.
\end{align*}
Taking $\psi=\phi_h/(k_BT)$ in \eqref{fem2}, we get 
$$
 \int_{\Omega}
    \sum_{i=1}^N
    \frac{z_i e}{k_BT} \exp(u_{h,i})
   \phi_h 
  \mathrm{dx}
  = 
 \int_{\Omega}
 \left( \frac{\epsilon}{k_BT}|\nabla\phi_h|^2
 -\frac{\rho_0}{k_BT}\phi_h\right)
  \mathrm{dx}.
$$
Combining the above two expressions, we get 
$$
  E_h= \int_{\Omega}\left(\sum_{i=1}^NU(u_{h,i})
   +\frac{\epsilon}{2k_BT}|\nabla\phi_h|^2
 \right)\mathrm{dx}.
$$
This completes the proof of \eqref{ener}.
\end{proof}

\subsection{Temporal discretization}
In this subsection, we discretize
the differential-algebraic system \eqref{fem} using a discontinuous Galerkin
(DG)
time integrator of degree $m\ge 0$. The major advantage of using the DG integrator is that we can
prove an unconditional energy stability result for the fully discrete scheme
for any polynomial degree $m\ge 0$ in time. To the best of our knowledge, this
is the first class of arbitrarily high-order accurate numerical schemes for the
PNP equations that are provably
positive and unconditionally energy stable.

At the $n$-th time level $t^n$, we denote the time step size as $\Delta t^n$ and the
update time interval as 
$I^n=[t^n,t^{n+1})$. We denote the space-time finite element space
on the space-time slab $\Omega\times I^n$ as follows:
\begin{align}
   \label{space-pkm}
   \Vst := \{v\in H^1(\Omega\times I^n):\; v|_{K\times
   I^n}\in\pol_k(K)\otimes \pol_m(I^n),\quad \forall K\in\Th\},
\end{align}
The fully discrete scheme then reads as follows:
for any $n=1,2,\cdots$, find 
$u_{h,i}^n\in \Vst$ and $\phi_h^n\in \Vst$
with 
$\int_{\Omega}\phi_h^{n,-}\,\mathrm{dx}=0$
such that 
\begin{subequations}
  \label{sfem}
  \begin{align}
    \label{sfem1}
   \int_{I^n}\int_{\Omega}{
   \frac{\partial \exp(u_{h,i}^n)}{\partial t}}
  v_i\,\mathrm{dx} \mathrm{dt}
+ \int_{\Omega}{\left(\exp(u_{h,i}^{n,+})-\exp(u_{h,i}^{n-1,-})\right)}
v_i^{n,+}
\,\mathrm{dx}
&
\nonumber\\
+\int_{I^n}\int_{\Omega} 
    D_i\exp(u_{h,i}^n)\left(\nabla u_{h,i}^n
   +\frac{z_i e}{k_B T}\nabla\phi_h^n
 \right)\cdot\nabla v_i\,\mathrm{dx}\mathrm{dt}
&=\;0,\quad 
\forall v_i\in \Vst,\\
    \label{sfem2}
    \int_{I^n}\int_{\Omega}
    \epsilon\nabla \phi_h^n\cdot\nabla\psi\,\mathrm{dx}\mathrm{dt} -\; 
    \int_{I^n}\int_{\Omega}
\left(\rho_0+\sum_{i=1}^Nz_ie\exp(u_{h,i}^n)\right)
\psi\,\mathrm{dx}\mathrm{dt} &=\;0,\quad \forall \psi\in \Vstx,\\
    \label{sfem3}
\int_{\Omega}
\epsilon\nabla \phi_h^{n,-}\cdot\nabla\psi^{n,-}\,\mathrm{dx} -\; 
\int_{\Omega}
\left(\rho_0+\sum_{i=1}^Nz_ie\exp(u_{h,i}^{n,-})\right)
\psi^{n,-}\,\mathrm{dx} &=\;0,\quad \forall \psi\in \Vh,
  \end{align}
\end{subequations}
where we denote $\xi^{n,-}:=\lim_{t\nearrow t^{n+1}}\xi^n(t, x)$
and $\xi^{n,+}:=\lim_{t\searrow t^{n}}\xi^n(t, x)$.
\begin{remark}
We note that the classical upwinding DG time integrator is used 
in \eqref{sfem1}, and a non-standard Gauss-Radau type quadrature 
rule is used
in \eqref{sfem2}--\eqref{sfem3} 
for the temporal discretization of the 
Poisson equation.
This Gauss-Radau type quadrature rule 
is necessary for us to prove the unconditional energy
stability of the scheme \eqref{sfem}, see the proof of Theorem \ref{thm:sfem} below.
We further note that when $m=0$, the temporal discretization is simply the
backward Euler method.
\end{remark}

The following results show that 
the fully-discrete scheme \eqref{sfem} satisfies a discrete version of 
three properties \eqref{prop}.

\begin{theorem}
  \label{thm:sfem}
  Assume $c_i^0>0$ for all $i$. 
  For any solution to the scheme \eqref{sfem}, 
  the following three properties holds
  \begin{subequations}
  \label{prop-s}  
  \begin{align}
    \label{prop-s1}
  \int_{\Omega}{\exp(u_{h,i}^{n,-})}\,\mathrm{dx}
 = &\;\int_{\Omega}{\exp(u_{h,i}^{n-1,-})}\,\mathrm{dx}, \\
    \label{prop-s2}
 c_{h,i}^n:= \exp(u_{h,i}^n)> &\; 0, \\
    \label{prop-s3}
 E_h^{n}-E_h^{n-1}=&\; 
 - \mathsf{Diss}_h^n -N_{h,1}^n-N_{h,2}^n 
\end{align}
where the discrete energy is given by  
\begin{align}
  \label{enerD}
  E_h^{n}:=
\int_{\Omega} \left( \sum_{i=1}^NU(u_{h,i}^{n,-})
+
      \frac{\epsilon}{2k_BT}|\nabla
      \phi_h^{n,-}|^2\right)
\,\mathrm{dx},
\end{align}
the (non-negative) physical dissipation term 
\begin{align*}
 \mathsf{Diss}_h^n
:=    \int_{I^n}  \int_{\Omega} 
\sum_{i=1}^ND_ic_{h,i}^n|\nabla\mu_{h,i}^n|^2 \,\mathrm{dxdt},
\end{align*}
and
the (non-negative) numerical dissipation terms $N_{h,1}^n$ and $N_{h,2}^n$ 
for the temporal discretization
are given as follows
\begin{align*}
  N_{1,h}^n :=&\; 
\int_{\Omega}
 \sum_{i=1}^N   
    \frac12\exp(\xi_{h,i}^n)(u_{h,i}^{n-1,-}-u_{h,i}^{n,+})^2
      \,\mathrm{dx},\\
  N_{2,h}^n:=&\;
\int_{\Omega}
\frac{\epsilon}{2k_BT}|\nabla
(\phi_h^{n,+}-\phi^{n-1,-})|^2
    \,\mathrm{dx},
\end{align*}
where $\xi_{h,i}^n$ is a function between 
$u_{h,i}^{n,+}$ and 
$u_{h,i}^{n-1,-}$ for each $i$.
\end{subequations}
\end{theorem}
\begin{proof}
  \begin{subequations}
    \label{enerX}
  Again, we only need to prove the energy stability result \eqref{prop-s3} as
  the other two properties are trivially satisfied.

We follow the same proof as the semi-discrete case in Theorem \ref{thm:fem}.
  Denoting 
  $c_{h,i}^n:=\exp(u_{h,i}^n)$, 
  $\mu_{h,i}^n:=u_{h,i}^n+\frac{z_i e}{k_B T}\phi_h^n\in \Vst$, 
  and taking test function $v_i=\mu_{h,i}^n$ in \eqref{sfem1}, we get
  \begin{align}
    \label{enerX0}
\int_{I^n}\int_{\Omega}\frac{\partial c_{h,i}^n}{\partial t} 
    \mu_{h,i}^n
    \,\mathrm{dxdt}
+ \int_{\Omega}{\left(c_{h,i}^{n,+}-c_{h,i}^{n-1,-}\right)}
\mu_{h,i}^{n,+}
\,\mathrm{dx}
    +
    \int_{I^n}  \int_{\Omega} 
    D_ic_{h,i}^n|\nabla\mu_{h,i}^n|^2 \,\mathrm{dxdt}
=&\;0.
  \end{align}
A simple calculation yields that  
  \begin{align*}
    \int_{I^n}
    \int_{\Omega}
    \frac{\partial c_{h,i}^n}{\partial t} 
    u_{h,i}
    \,\mathrm{dxdt} 
    =&\; \int_{I^n}\frac{d}{dt}
    \int_{\Omega}U(u_{h,i})\,\mathrm{dxdt}
    =\; 
    \int_{\Omega}U(u_{h,i}^{n,-})\,\mathrm{dx}
    -\int_{\Omega}U(u_{h,i}^{n,+})\,\mathrm{dx},
  \end{align*}
  where $U(\eta):= \exp(\eta)(\eta-1)$.
  By Taylor expansion we have
  $$
  (\exp(a)-\exp(b))a = U(a)- U(b)+\frac{1}{2}\exp(\xi)(a-b)^2
  $$
  for some $\xi$ between $a$ and $b$.
Hence, 
  \begin{align*}
  \int_{\Omega}{\left(c_{h,i}^{n,+}-c_{h,i}^{n-1,-}\right)}u_{h,i}^{n,+}\mathrm{dx}
  =
    \int_{\Omega}
    \left(U(u_{h,i}^{n,+})
    -U(u_{h,i}^{n-1,-})
    +\frac12\exp(\xi_{h,i}^n)(u_{h,i}^{n-1,-}-u_{h,i}^{n,+})^2
    \right)
      \,\mathrm{dx}
\end{align*}
for some $\xi_{h,i}^n$ between $u_{h,i}^{n-1,-}$ and $u_{h,i}^{n,+}$.
Combining the above two equalities
and summing the terms \eqref{enerX0} over all the indices $i$, 
we get
\begin{align}
  \label{enerX2}
  \sum_{i=1}^N
  \left( \int_{\Omega}U(u_{h,i}^{n,-})\,\mathrm{dx}
  -\int_{\Omega}U(u_{h,i}^{n-1,-})\,\mathrm{dx}\right)
 +N_{1,h}^n  +I_{c,\phi}
    +
   \underbrace{ \int_{I^n}  \int_{\Omega} 
   \sum_{i=1}^ND_ic_{h,i}^n|\nabla\mu_{h,i}^n|^2
 \,\mathrm{dxdt}}_{=\mathsf{Diss}_h^n}
    = 0,
\end{align}
where 
$N_{1,h}^n$ is the following numerical dissipation term from upwinding
$$
N_{1,h}^n := 
\int_{\Omega}
 \sum_{i=1}^N   
    \frac12\exp(\xi_{h,i}^n)(u_{h,i}^{n-1,-}-u_{h,i}^{n,+})^2
      \,\mathrm{dx}\ge 0,
$$
and
the term $I_{c,\phi}$ is given as follows:
\begin{align*}
  I_{c,\phi}
 :=
\int_{I^n}\int_{\Omega}\frac{\partial
S_{h}^n}{\partial t} 
    \phi_h^n
    \,\mathrm{dxdt}
+ \int_{\Omega}{\left(S_{h}^{n,+}-S_{h}^{n-1,-}\right)}
\phi_{h,i}^{n,+}
\,\mathrm{dx},
\end{align*}
with 
$
S_h^{n}:=\sum_{i=1}^N\frac{z_i e }{k_BT}c_{h,i}^n+\frac{\rho_0}{k_BT}
$ being the source term in \eqref{sfem2}, 
where we used the fact that $\rho_0, z_i, k_B, T$ are independent of time.
Next, taking $\psi=\partial_t\phi_{h}^n/(k_BT)\in \Vstx$ in equation \eqref{sfem2}
and combining with the above term, we get
\begin{align*}
  I_{c,\phi}=&\;-
    \int_{I^n}\int_{\Omega}
    \frac{\partial }{\partial t} (\frac{\epsilon}{2k_BT}|\nabla
    \phi_h^n|^2)\,\mathrm{dx}\mathrm{dt} +
    \int_{I^n}\int_{\Omega}
    \frac{\partial}{\partial t}(S_h^n
    \phi_h^n)\,\mathrm{dx}\mathrm{dt} 
+ \int_{\Omega}{\left(S_{h}^{n,+}-S_{h}^{n-1,-}\right)}
\phi_{h,i}^{n,+}
\,\mathrm{dx},\\
=&\;
    \int_{\Omega}
    \left(
      -\frac{\epsilon}{2k_BT}|\nabla
    \phi_h^{n,-}|^2
    +
\frac{\epsilon}{2k_BT}|\nabla
    \phi_h^{n,+}|^2
    +S_h^{n,-}\phi_h^{n,-}
    -
    S_h^{n-1,-}\phi_h^{n,+}
  \right)
    \,\mathrm{dx},\\
=&\;
    \int_{\Omega}
    \left(
    S_h^{n,-}\phi_h^{n,-}
      -\frac{\epsilon}{2k_BT}|\nabla
    \phi_h^{n,-}|^2
    +
\frac{\epsilon}{2k_BT}|\nabla
(\phi_h^{n,+}-\phi^{n-1,-})|^2
-
\frac{\epsilon}{2k_BT}|\nabla
    \phi_h^{n-1,-}|^2\right)\,\mathrm{dx}\\
 &\;
 +\underbrace{\int_{\Omega}\left(
   \frac{\epsilon}{2k_BT}\nabla\phi_h^{n-1,-}-
S_h^{n-1,-})\phi_h^{n,+}
  \right)
\,\mathrm{dx}}_{=0 \text{ by \eqref{sfem3}.}},\\
=&\;
    \int_{\Omega}
    \left(
      \frac{\epsilon}{2k_BT}|\nabla
    \phi_h^{n,-}|^2
    +
\frac{\epsilon}{2k_BT}|\nabla
(\phi_h^{n,+}-\phi^{n-1,-})|^2
-
\frac{\epsilon}{2k_BT}|\nabla
    \phi_h^{n-1,-}|^2\right)\,\mathrm{dx},
\end{align*}
where the last step is again due to \eqref{sfem3} by taking 
the test function 
$\psi=\phi_{h}^{n,-}/(k_BT)$.
Combining the above expression with the equality \eqref{enerX2}, 
we finally get 
\begin{align*}
  E_h^n
+N_{1,h}^n  +N_{2,h}^n
+\mathsf{Diss}_h^n
= E_h^{n-1},
\end{align*}
where we used the definition of the energy $E_h^n$
and the numerical dissipation term $N_{2,h}^n$ from Theorem \ref{thm:sfem}.
This completes the proof.
  \end{subequations}
\end{proof}

\begin{remark}
  \label{rk:3}
 Our proposed scheme \eqref{sfem} can be also applied to the variable 
 coefficient case to achieve all the three properties in Theorem \ref{thm:sfem}.
 For example, by multiplying all the spatial integrals in \eqref{thm:sfem}
 with a 
 coefficient $A(x)>0$, we 
 immediately get a positive and unconditionally energy stable scheme for 
 the following PNP system with a variable (possibly discontinuous) 
$A(x)$:
\begin{subequations}
  \label{pnpX}
  \begin{align}
  \label{pnp1X}
   A \frac{\partial c_i}{\partial t} = &\; \nabla\cdot \left(A
  D_i c_i\nabla\mu_i
  \right),\quad i=1,2,\cdots, N,\\
  \label{pnp2X}
     -\nabla\cdot(A\epsilon\nabla \phi) = &\; A(\rho_0+\sum_{i=1}^Nz_iec_i).
  \end{align}
  In this case, all the integrals in \eqref{prop-s} shall be weighted by the 
  cross-section $A(x)$.
\end{subequations}
\end{remark}

\subsection{Nonlinear system solver and adaptive time stepping}
We use Newton's method to solve the fully discrete 
nonlinear system \eqref{sfem}. A sparse direct solver is used for the 
linear algebraic systems in each Newton's iteration.
The Newton's iteration is stopped when the relative error of the energy 
\eqref{enerD} is reduced by a factor of $10^{-8}$.

One main advantage of unconditionally energy stable schemes with arbitrary order is that they can
be easily implemented with an adaptive time stepping strategy so that the time step
is dictated only by accuracy rather than by stability as with conditionally stable
schemes.
Here we use the classical PI step size control algorithm
\cite{Gustafsson88}:
given a previous time step size $\Delta t^{n-1}$, 
the next time step  size $\Delta t^{n}$ is proposed as
\begin{align}
  \label{stepsize}
  \Delta t_{temp}=&\;
  \left(\frac{tol}{e_{n}}\right)^{K_I}\left(\frac{e_{n-1}}{e_n}\right)^{K_P}\Delta
  t^{n-1},\quad\quad
  \Delta t^{n}=\min\{\Delta t_{temp}, \theta_{\max}\Delta t^{n-1}, \Delta
  t_{\max}\},
\end{align}
where $tol$ is a prescribed error tolerance, 
$\Delta t_{\max}$ is a user defined maximum allowed time step size,
and
\begin{align}
  \label{error-est}
  e^n := \left|\frac{E_h^{n}-E_h^{n,lo}}{E_h^{n}}\right|
\end{align}
is the error estimator based on the relative error in energy between
the scheme \eqref{sfem} (with temporal order $m\ge 1$) and a companion (temporal) low-order scheme
(with $m=0$) with the same spatial discretization, and 
we use the following parameters suggested in \cite{Gustafsson88}:
\begin{align*}
  K_P=0.13,\quad K_I=1/15,\quad \theta_{\max}=2, \quad \rho=1.2.
\end{align*}
We reject the proposed $\Delta t^n$ in \eqref{stepsize} if either the Newton iteration did not
converge or the target tolerance is violated ($e^n> \rho\, tol$).
In this case, we halve the time step size by setting $\Delta t^{n}:=0.5\Delta
t^{n-1}$ and redo the computation.

\section{Numerics}
\label{sec:num}
In this section, 
we present several numerical examples to validate our theoretical results in the
previous section.
Our numerical simulations are performed using the open-source finite-element software 
{\sf NGSolve} \cite{Schoberl16}, \url{https://ngsolve.org/}.
In particular, {\sf NGSolve}'s add-on library {\sf ngxfem},
\url{https://github.com/ngsxfem/ngsxfem},
is used in the implementation of the space-time finite element scheme
\eqref{sfem}.

\subsection*{Example 1}(Accuracy test)
We first use a manufactured solution example to test the accuracy of the scheme 
\eqref{sfem}.
We consider the PNP equations \eqref{pnp} with 
$N=2$, $z_1=1$, $z_2=-1$, $e=k_B=T=D_1=D_2=\epsilon=1$.
The compuational domain is a unit square, and we take soure terms such that 
the smooth exact solution is
\begin{align}
  c_1(t, x, y) = &\; 1+0.5 \sin(t)\sin(\pi x)\sin(\pi y),\nonumber\\
  \label{ref-soln}
  c_2(t, x, y) = &\; 1-0.5 \sin(t)\sin(\pi x)\sin(\pi y),\\
  \phi(t, x, y) =&\; \sin(t)\sin(\pi x)\sin(\pi y).\nonumber
\end{align}
We use homogeneous Dirichelt boundary conditions on $u_1, u_2$ and $\phi$.

We apply the space-time finite element scheme \eqref{sfem} with $k=m=1$, $k=m=2$, and 
$k=m=3$ on a sequence of uniformly refined trianglar meshes.
We take uniform time step size $\Delta t = 2h$, where $h$ is the mesh size.
The Newton's method converges within $3$--$4$ iterations for all the cases.
The $L^2$-errors at time $t=1$ are recorded in Table~\ref{table:1X}. 
Clearly, we observe the expected ($k+1$)-th order of
convergence for all the variables 
for the scheme \eqref{sfem} using polynomials of degree $k=m$.
\begin{table}[ht!]
\begin{center}
  \scalebox{1}
  {
  \begin{tabular}{ c|c | c  c|c c | c c} 
  $k=m$ & $1/h$ & $L^2$-err in $u_h^1$ & rate &
    $L^2$-err in $u_h^2$ & rate &
    $L^2$-err in $\phi_h^2$ & rate\\ 
    \hline
&  8 & 5.228e-03 & -- & 1.362e-02 & -- & 2.140e-02 & --\\
1& 16 & 1.473e-03 & 1.828 & 3.821e-03 & 1.834 & 5.461e-03 & 1.970\\
& 32 & 3.923e-04 & 1.908 & 1.000e-03 & 1.934 & 1.374e-03 & 1.991\\
& 64 & 1.016e-04 & 1.950 & 2.552e-04 & 1.970 & 3.442e-04 & 1.997\\
\hline
&  8 & 1.218e-04 & -- & 2.615e-04 & -- & 1.999e-04 & --\\
2& 16 & 1.407e-05 & 3.113 & 1.928e-05 & 3.762 & 1.762e-05 & 3.504\\
& 32 & 1.702e-06 & 3.048 & 1.782e-06 & 3.435 & 1.869e-06 & 3.237\\
& 64 & 2.106e-07 & 3.014 & 1.992e-07 & 3.162 & 2.215e-07 & 3.077\\
\hline
&  8 & 9.026e-06 & -- & 1.980e-05 & -- & 1.809e-05 & --\\
3& 16 & 4.813e-07 & 4.229 & 1.115e-06 & 4.151 & 1.083e-06 & 4.062\\
& 32 & 2.768e-08 & 4.120 & 6.665e-08 & 4.064 & 6.651e-08 & 4.025\\
& 64 & 1.675e-09 & 4.047 & 3.990e-09 & 4.062 & 4.122e-09 & 4.012\\
\hline
  \end{tabular}
}
\end{center}
\caption{
  \textbf{Example 1.}
  $L^2$-errors at time $t=1$ 
  for the methods \eqref{sfem} with $k=m$ on a sequence of uniformly refined meshes. 
  Time step size $\Delta t = 2h$.
}
\label{table:1X}
\end{table}


\subsection*{Example 2}(One-dimensional problem with discontinuous coefficients)
Here we solve a two-component ($N=2$) one-dimensional system \eqref{pnpX} with variable
discontinuous coefficients.
The domain is $\Omega=[-28, 25]$, and we use the homogeneous Dirichlet boundary conditions 
$$
\phi(x)=u_1(x)=u_2(x) = 0 \quad \forall x\in \partial \Omega
$$
and initial condition 
$$
u_1(0, x)=u_2(0, x) = 0.
$$
We use the following set of parameters:
\begin{align*}
  k_B=&\;T=e=1, z_1=1, z_2 = -1,
  D_1 = 1, D_2=1.0383,\\
  A = \pi r^2, \text{ with } 
  r(x)=&\;\left\{
    \begin{tabular}{ll}
      $-0.5x-7$ &if $-28<x<-18$,\\[0.5ex]
      $2$ &if $-18<x<-5$,\\[0.5ex]
      $0.5$ &if $-5<x<10$,\\[0.5ex]
      $0.9x-8.5$ &if $10<x<25$,
\end{tabular}
  \right.\\
      \epsilon(x)=&\;\left\{
    \begin{tabular}{ll}
      $4.7448$ &if $-5<x<-10$,\\[0.5ex]
      $189.79$ &elsewhere,
\end{tabular}
\right.\\
      \rho_0(x)=&\;\left\{
    \begin{tabular}{ll}
      $-300$ &if $x\in(-2,-1)\cap (0,1)\cup (2,3)\cup(4,5)\cup (6,7)$,\\[0.5ex]
      $0$ &elsewhere.
\end{tabular}
\right.
\end{align*}
We apply the second-order scheme \eqref{sfem} with polynomial degree $k=m=1$, 
and take the initial time stepsize as 
$\Delta t^1 = 10^{-4}$. 
The adaptive time stepping algorithm \eqref{stepsize} is used where
the companion  low order scheme takes $k=1$ and $m=0$ on the same mesh.
For the two user defined paramters in \eqref{stepsize}, we take 
$tol=10^{-3}$, and $\Delta t_{\max}= \left\{
  \begin{tabular}{ll}
    $2$ & if  $t<250$,\\[1ex]
    $200$ & if  $t>250$.
\end{tabular}
\right.
$
The simulation is terminated when relative error in the energy
in two consecutive time steps is less than $10^{-13}$,  
$$
\left|\frac{E_h^n-E_h^{n-1}}{E_h^{n}}\right|< 10^{-13},
$$
which indicates a steady state is reached. 
We find the steady state is reached around time $t=1400$.

This problem is very challenging to solve as $c_2 = \exp(u_2)$ stays near {\it
zero} for  $x\in (-5, 10)$ for an extended period of time. Typical solutions at
different times are shown in Figure~\ref{fig:1}--\ref{fig:3}, which is obtained
by the scheme \eqref{sfem} on a very fine uniform mesh with mesh size $h=1/128$.
It is clear from Figure~\ref{fig:3} that $c_2 = \exp(u_2)$
stays positive and below $\exp(-50)\approx 2\times 10^{-22}$ for $x\in (-5, 10)$
and $t\in (20, 240)$, which can be as small as 
$\exp(-150)\approx 7\times 10^{-66}$ at around time $t=100$.
A non positivity preserving scheme may easily lead to a negative
density $c_2$, hence an early termination of the code using such scheme due to the need to evaluate $\log(c_2)$.
\begin{figure}[h!]
\centering
\includegraphics[width=0.9\textwidth]{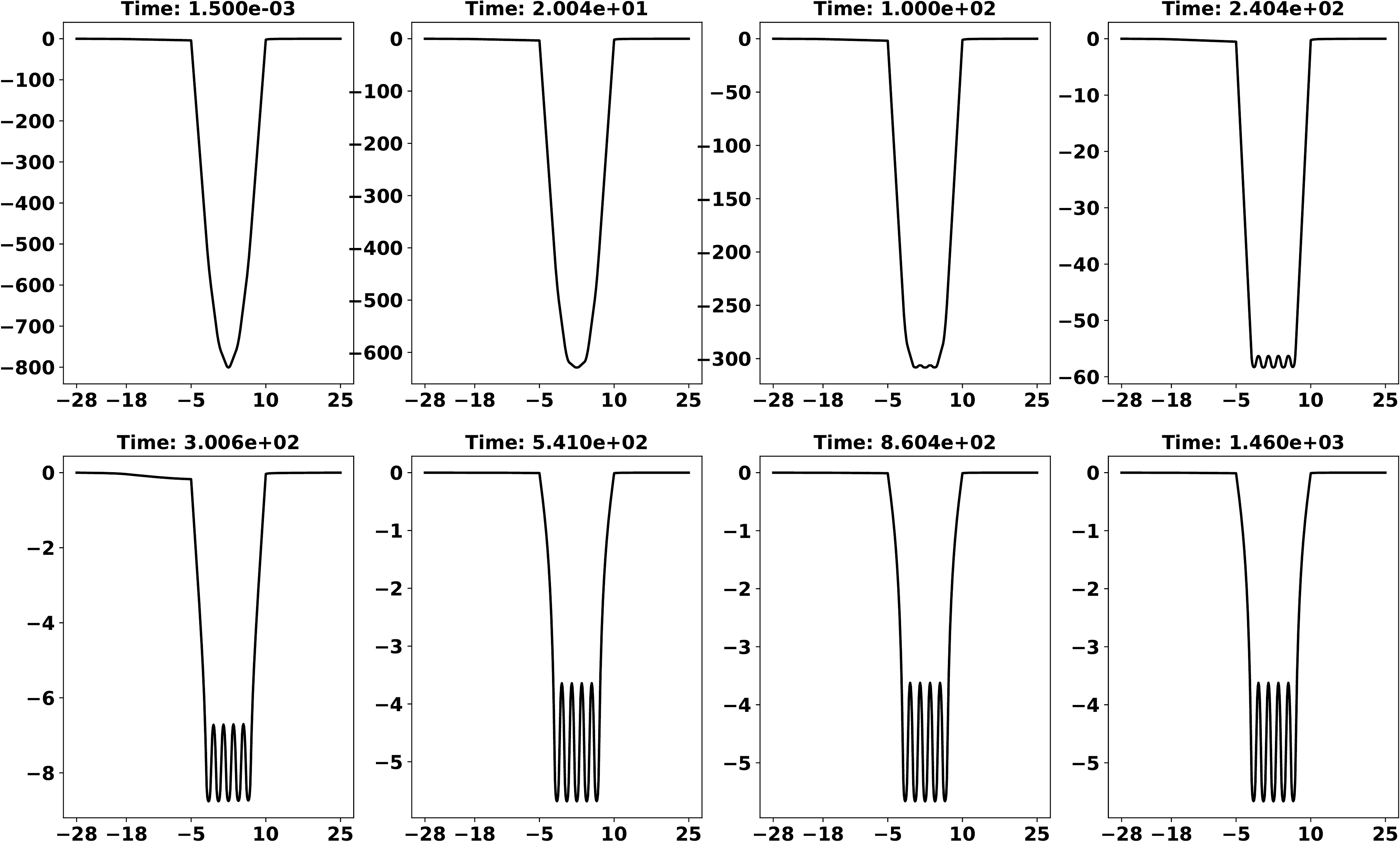}
\caption{
  \textbf{Example 2.}
  Electrostatic potential $\phi$ at different times obtained with 
  the scheme \eqref{sfem} with $k=m=1$ on a uniform mesh with mesh size $h = 1/128$.}
\label{fig:1}
\end{figure} 

\begin{figure}[h!]
\centering
\includegraphics[width=0.9\textwidth]{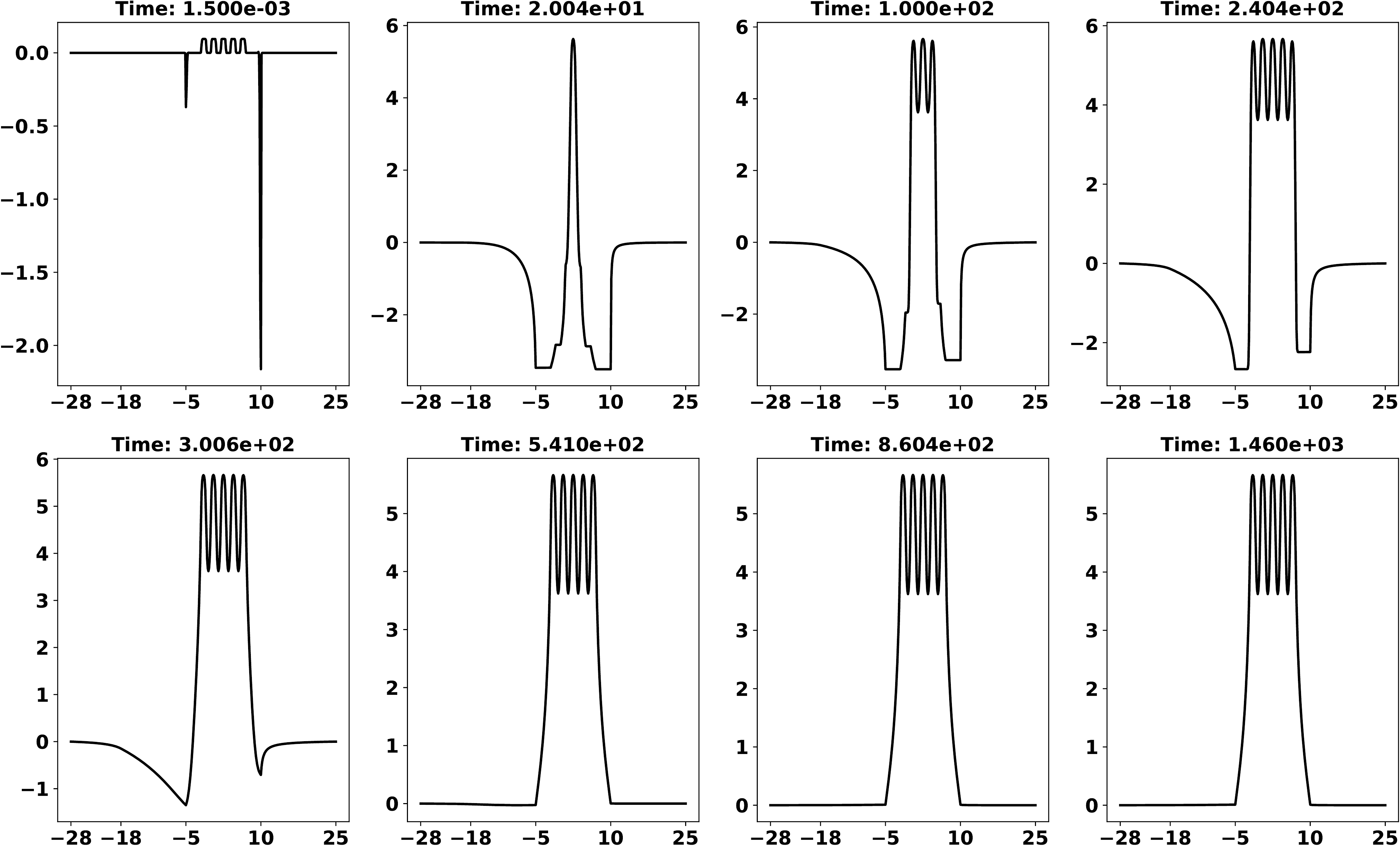}
\caption{
  \textbf{Example 2.}
  Logarithmic density $u_1 = \log(c_1)$ at different times obtained with 
  the scheme \eqref{sfem} with $k=m=1$ on a uniform mesh with mesh size $h = 1/128$.}
\label{fig:2}
\end{figure} 

\begin{figure}[h!]
\centering
\includegraphics[width=0.9\textwidth]{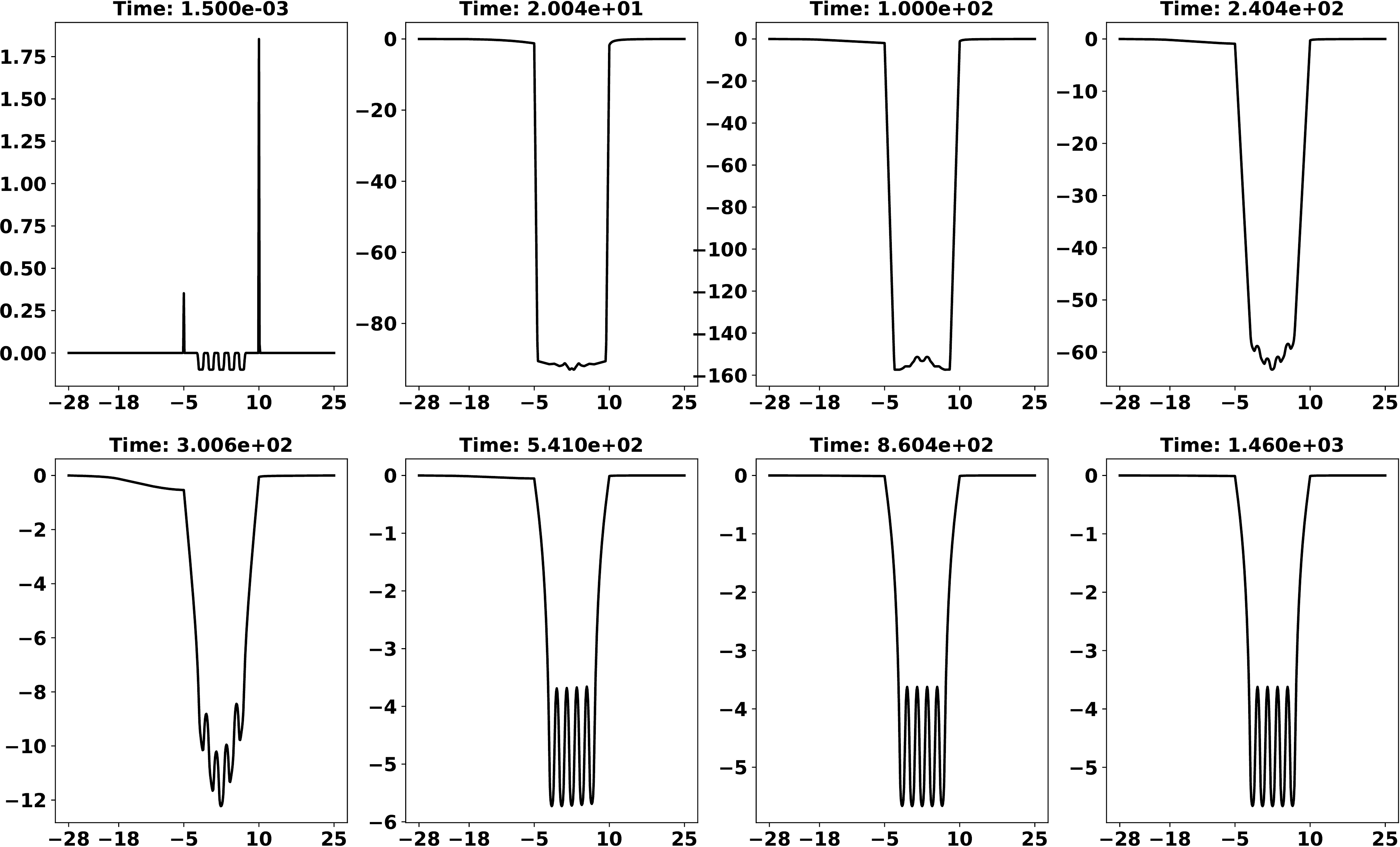}
\caption{
  \textbf{Example 2.}
  Logarithmic density $u_2 = \log(c_2)$ at different times obtained with 
  the scheme \eqref{sfem} with $k=m=1$ on a uniform mesh with mesh size $h = 1/128$.}
\label{fig:3}
\end{figure}

The initial and final energy, along with the total number of 
time steps to reach the steady state is recorded in Table~\ref{table:1}
for four consecutively refined (uniform) meshes. 
We clearly observe a convergence in the energies as the mesh is refined.
Moreover, the total number of time steps to reach the steady state is similar on all four meshes.
\begin{table}[ht!]
\begin{center}
  \scalebox{1}
  {
  \begin{tabular}{ c | c | c|c|c} 
    mesh size $h$ & $1/16$ & $1/32$ & $1/64$ & $1/128$\\
    \hline
    initial    energy & 
    387788.75 &
    387798.52&
    387800.97&
    387801.58\\
    \hline
    final    energy & 
    -3023.3435&
    -3022.3990&
    -3022.1619&
    -3022.1025
    \\
    \hline
    total time steps &
    206 & 208 & 210 & 210
  \end{tabular}
}
\end{center}
\caption{
  \textbf{Example 2.}
  Initial and final energy, and total number of time steps 
  to reach the steady state for the scheme \eqref{sfem} with $k=m=1$
  and adaptive time stepping \eqref{stepsize} on different meshes.}
\label{table:1}
\end{table}

We plot in Figure~\ref{fig:conv} the energy evolution 
over time on the four meshes, along with the evolution of the dissipation rates 
$$
-\frac{E_h^{n}-E_h^{n-1}}{\Delta t^n}, \text{ and }
\mathsf{Diss}_h^n/\Delta t^n,
$$ 
The results in Figure~\ref{fig:conv} numerically confirmed the energy stability
result in Theorem \ref{thm:sfem}. Moreover, the energy and dissipation rate
evolution are very similar on the four meshes, and we also observe that 
the computed  dissipation rate $(E_h^{n-1}-E_h^n)/\Delta t^n$ is very close to
and slightly larger than  
the physical dissipation rate $\mathrm{Diss}_h^n/\Delta t^n$, which is again
consistent with the equality \eqref{prop-s3} in Theorem \ref{thm:sfem}.
\begin{figure}[h!]
\centering
\includegraphics[width=0.9\textwidth]{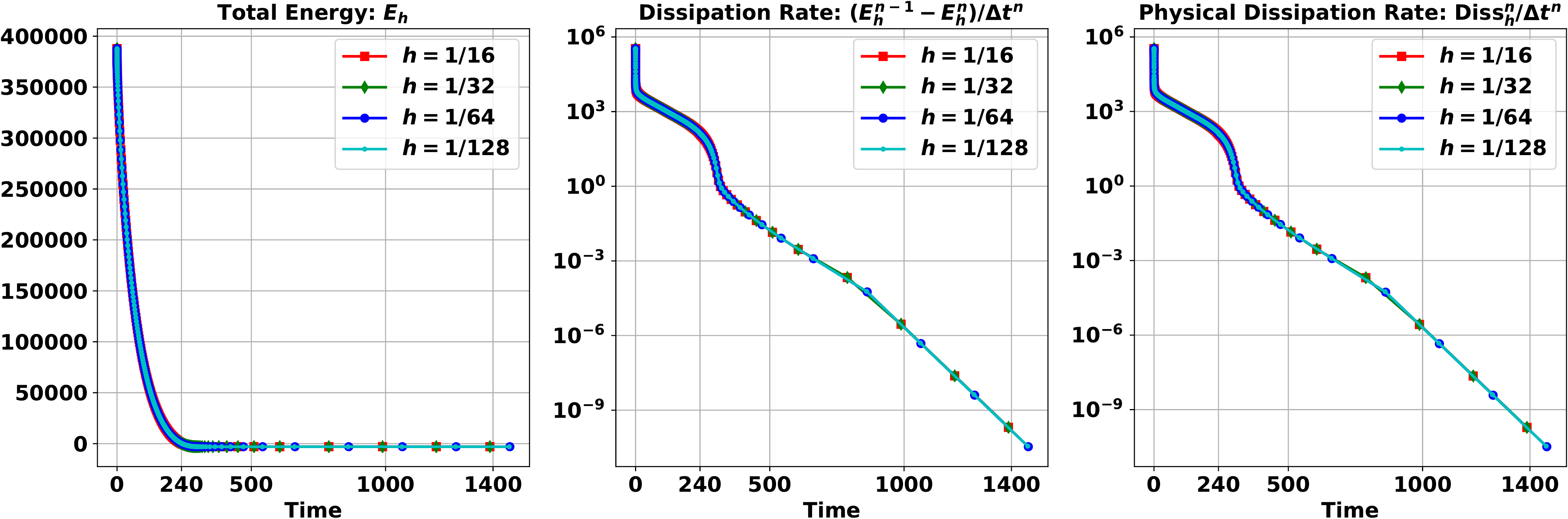}
\caption{
  \textbf{Example 2.} 
  The evolution of energy and dissipation rates
over time on the four meshes.}
\label{fig:conv}
\end{figure} 

Finally we plot in Figure~\ref{fig:conv2} the evolution of the time step size, the error
estimator 
$e^n = \left|(E_h^n-E_h^{n,lo})/{E_h^n}\right|$,  
and the number of Newton iterations. Again, the results on the four meshes are
very 
similar to each other, and the average number of Newton iterations
is about $3-4$. In particular, we observe that 
the time step size gradually increases from $\Delta t=10^{-4}$
to $\Delta t=\Delta t_{\max}=2$ till time $t=14$, then it stays at 
$\Delta t_{\max}=2$ for a period of time till around $t=235$, where
a few drops in time step size occur from $t = 235$ to $t=243$ due to a relative large error $e^n$.
Afterwards, $\Delta t$ gradually increases to $\Delta t_{\max}=200$.
The advantage of adaptive time stepping is also clearly illustrated in Figure~\ref{fig:conv2}
as a scheme with a uniform time step size $\Delta t = \Delta t^1=10^{-4}$ would
need about $1.4\times 10^7$ time steps to converge to steady state, while our adaptive 
time stepping scheme 
only needs about $200$ time steps as shown in Table~\ref{table:1}.

\begin{figure}[h!]
\centering
\includegraphics[width=0.9\textwidth]{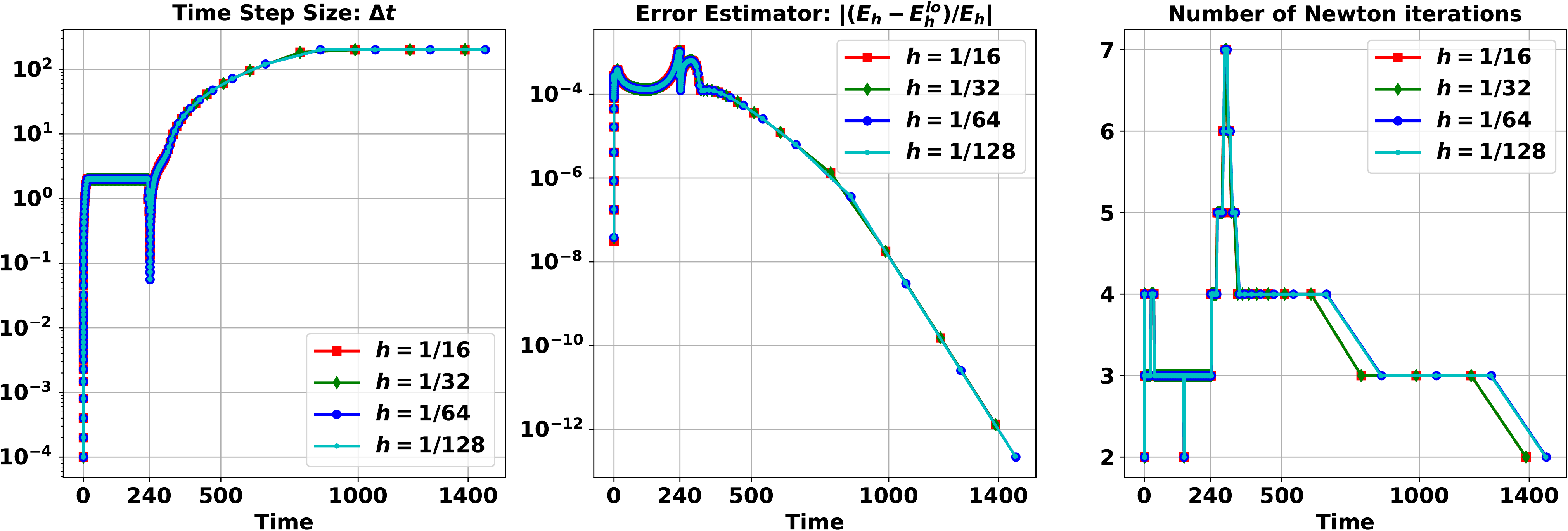}
\caption{
  \textbf{Example 2.} 
  The evolution of time step size $\Delta t^n$, error estimator $e^n$, and 
  Newton iteration number 
over time on the four meshes.}
\label{fig:conv2}
\end{figure}

\subsection*{Example 3}(Two-dimensional problem with discontinuous coefficients)
Here we solve a problem similar to Example 2 on a two-dimensional geometry.
The computational domain is shown in Figure~\ref{fig:2d}.
\begin{figure}[ht]
\centering
  \begin{tikzpicture}
    
 
    \draw[ultra thick,draw=black, pattern=north west lines] 
    (25/4, 0)
    to (25/4, 14/4)
    to (10/4, 0.5/4)
    to (-5/4, 0.5/4)
    to (-5/4,  2/4)
    to (-18/4,2/4) 
    to (-28/4,7/4)
    to (-28/4,0)
    to (25/4, 0)
    ;
    
  \node at (-28/4,7/4)[above,scale=1.2] {\tiny A};
  \node at (-18/4,2/4)[above,scale=1.2] {\tiny B};
  \node at (-5/4,2/4)[above,scale=1.2] {\tiny C};
 \node at (-4.3/4,0.1/4)[above,scale=1.2] {\tiny D};
  \node at (9.8/4,0.1/4)[above,scale=1.2] {\tiny E};
  \node at (25/4,14/4)[above,scale=1.2] {\tiny F};
  \node at (-28/4,0)[below,scale=1.2] {\tiny O};
  \node at (25/4,0)[below,scale=1.2] {\tiny P}; 
\end{tikzpicture}
\caption{
  \textbf{Example 3.} 
  Computational domain $\Omega$ (polygon).
Coordinates for the vertices: 
$A:(-28, 7)$, 
$B:(-18, 2)$, 
$C:(-5, 2)$, 
$D:(-5, 0.5)$, 
$E:(10, 0.5)$, 
$F:(25, 14)$, 
$O:(-28, 0)$, 
$P:(25, 0)$. 
  }
  \label{fig:2d}
\end{figure}
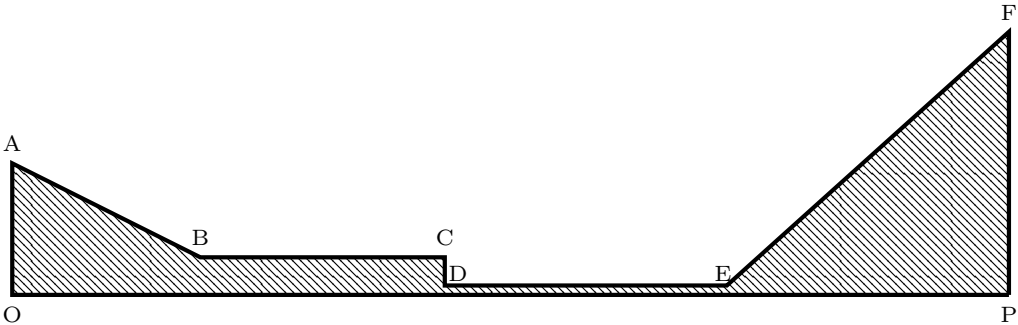
We solve the variable-coefficient PNP equations \eqref{pnpX} on the domain $\Omega$ 
using the same set of parameters as in Example 2, with the only exception  that 
the cross-section term is taken to be
$
A(x)=\pi r(x)
$
to take into account the 2D geometry.
Homogeneous Dirichlet boundary conditions 
are imposed on the left and right boundaries segments $OA$ and $PF$, and homogeneous Neumann
boundary conditions are imposed on the rest of the domain boundary.

The same second-order scheme \eqref{sfem} 
with polynomial degree $k=m=1$ and adaptive time stepping \eqref{stepsize}
is used here on three uniform unstructured meshes.
The coarse mesh with mesh size $h=1/16$ has 
$1.10\times 10^5$ elements which leads to a total of  $3.37\times 10^5$ degrees of freedom (DOFs).
The medium mesh with mesh size $h=1/32$ has 
$4.43\times 10^5$ elements which leads to a total of  $1.34\times 10^6$ DOFs.
The fine mesh with mesh size $h=1/64$ has 
$1.77\times 10^6$ elements which leads to a total of  
$5.34\times 10^6$ DOFs.
We use this example to show the performance of our method in a challenging
two-dimensional problem with variable coefficient and complex geometry.


The computational results are very similar to the 1D case in Example 2.
Hence, we only present in Figure~\ref{fig:2d3} the evolution of $u_2=\log(c_2)$
at different times along the center cut line $y=0$. 
In particular, we still observe that $c_2 = \exp(u_2)$ stays near {\it
zero} for  $x\in (-5, 10)$ for an extended period of time.

%

\begin{figure}[h!]
\centering
\includegraphics[width=0.9\textwidth]{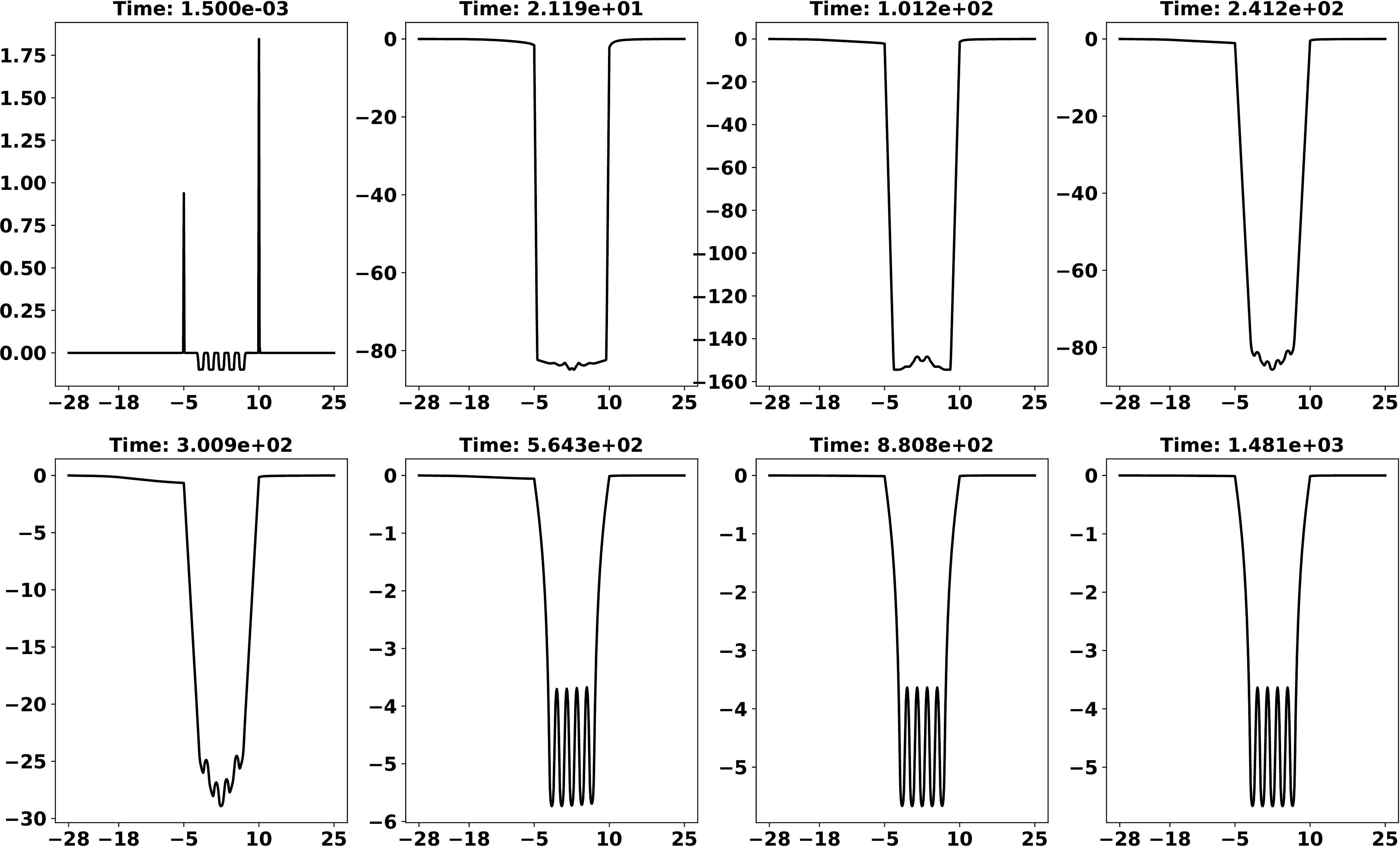}
\caption{
  \textbf{Example 3.} 
  Logarithmic density $u_2 = \log(c_2)$ along 
cut line $y=0$
  at different times obtained with 
  the scheme \eqref{sfem} with $k=m=1$ on a uniform unstructured mesh with mesh size $h
= 1/32$ ($4.43\times 10^5$ triangular elements).}
\label{fig:2d3}
\end{figure} 

We present in Table~\ref{table:2d1} the 
initial and final energy and the total number of 
time steps to reach the steady state, 
in  Figure~\ref{fig:2dconv} the energy evolution 
and the evolution of the dissipation rates,
and 
in Figure~\ref{fig:2dconv2} the evolution of the time step size, the error
estimator 
and the number of Newton iterations. 
Again, all the results are very similar to the 1D case in Example 2.
\begin{table}[ht!]
\begin{center}
  \scalebox{1}
  {
  \begin{tabular}{ c | c | c|c} 
    mesh size $h$ & $1/16$ & $1/32$  & $1/64$ \\
    \hline
    initial    energy & 
    402624.10 &
    395346.31 &
391737.06
    \\
    \hline
    final    energy & 
    -2918.1776 &
    -2969.7288 &
    -2995.4971
    \\
    \hline
    total time steps &
    216 & 213 & 212
  \end{tabular}
}
\end{center}
\caption{
  \textbf{Example 3.} 
  Initial and final energy, and total number of time steps 
  to reach the steady state for the scheme \eqref{sfem} with $k=m=1$
  and adaptive time stepping \eqref{stepsize} on different meshes.}
\label{table:2d1}
\end{table}

\begin{figure}[h!]
\centering
\includegraphics[width=0.9\textwidth]{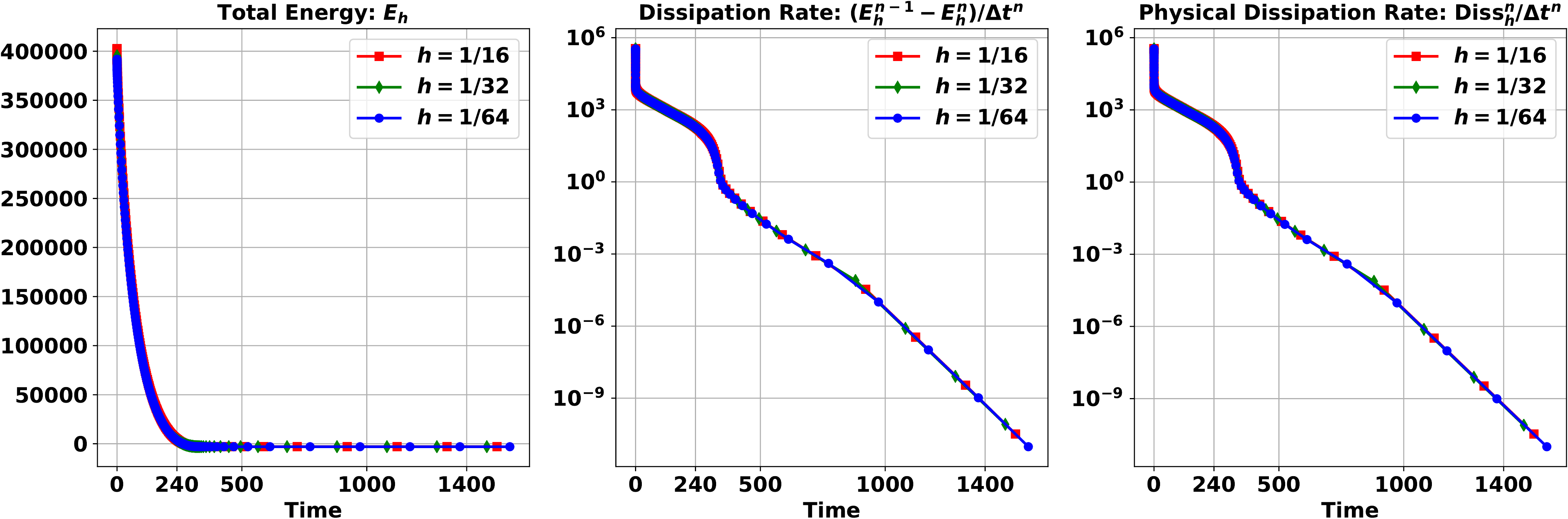}
\caption{
  \textbf{Example 3.} 
  The evolution of energy and dissipation rates
over time on the four meshes.}
\label{fig:2dconv}
\end{figure}

\begin{figure}[h!]
\centering
\includegraphics[width=0.9\textwidth]{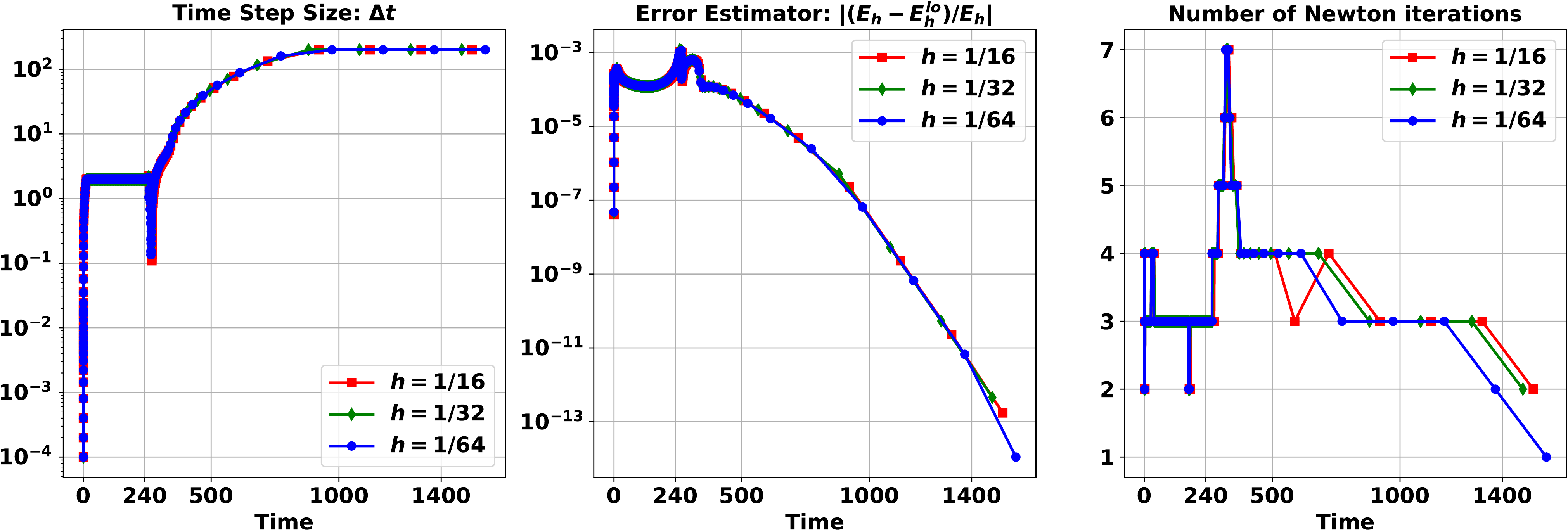}
\caption{
  \textbf{Example 3.} 
  The evolution of time step size $\Delta t^n$, error estimator $e^n$, and 
  Newton iteration number 
over time on the four meshes.}
\label{fig:2dconv2}
\end{figure} 

\section{Conclusion}
\label{sec:conclude}
We presented a novel class of high-order accurate, positivity preserving, and unconditionally 
energy stable
  space-time finite element schemes for the PNP  equations based on discretizing the {\it entropy variables} associated to the
  densities.
  To the best of our knowledge, this is the first class of arbitrarily
  high-order accurate schemes
  for PNP equations that is both positivity preserving and unconditionally
  energy stable.

  Our ongoing work consists of extending the STFEM framework to design
  positivity preserving and 
  unconditionally energy stable schemes for 
  other Wasserstein gradient flow problems, 
  and their coupling with incompressible
  flows like  electrokinetic problems. We are also planning to investigate on alternative
  finite element discretizations, efficient linear system solvers, and more computationally efficient temporal
  discretizations for the PNP equations in the future.

  \

  \textbf{Acknowledgement: } We would like to thank Christoph Lehrenfeld for
  suggesting using the space-time framework in the  {\sf ngsxfem} software.

  \bibliographystyle{siam}

\end{document}